\documentclass{amsart}%
\usepackage{amssymb}
\usepackage{amsfonts}
\usepackage{amsmath}
\usepackage{verbatim}
\usepackage{mathtools}
\usepackage{graphicx}%

\newtheorem{theorem}{Theorem}
\theoremstyle{plain}

\newtheorem{conjecture}{Conjecture}

\newtheorem{definition}{Definition}

\newtheorem{lemma}{Lemma}

\newtheorem{remark}{Remark}

\numberwithin{equation}{section}

\begin{document}
\title[ ]{Gap Phenomenon for Yamabe Type Problems of $M^m\times T^{n-m}$}
\author{Fang Wang}
\address{School of Mathematical Sciences, Shanghai Jiao Tong University, Shanghai 200240,}
\email{fangwang1984@sjtu.edu.cn}
\author{Zhixin Wang}
\address{School of Mathematical Sciences, Shanghai Jiao Tong University, Shanghai 200240,}
\email{jhin@sjtu.edu.cn}

\maketitle

\begin{abstract}
We prove that if $(M^m, h)$ is a Yamabe metric, then the product metric $h + g_{\mathrm{flat}}$ on $M^m \times T^{n-m}$ is also a Yamabe metric whenever the flat torus $T^{n-m}$ is sufficiently small. This generalizes earlier results for $S^1 \times S^{n-1}$. Our method extends to the study of Type~I and Type~II Yamabe constants, $Q$-curvature problems, and an isoperimetric-ratio type problem.
\end{abstract}

\section{Introduction}

For a general $(M^m,h)$ without boundary, its Yamabe constant $Y(M,h)$ and Yamabe invariant $\sigma(M)$ are defined by the
\begin{equation}
\begin{aligned}
    E_{(M,h)}(u)=&\frac{\int_M c_m|\nabla_M u|^2+Ru^2}{(\int_{M} u^{\frac{2m}{m-2}})^{\frac{m-2}{m}}}\\
    Y(M,[h])=&\inf\limits_{u\in H^1(M),u\neq 0} E_{(M,h)}(u)\\
    \sigma(M)=&\sup\limits_{h}Y(M,[h])
    \label{0-Yamabe-func}
\end{aligned}
\end{equation}
where $R$ is the scalar curvature of $g$ and $c_m=\frac{4(m-1)}{m-2}$. The Euler-Lagragian equation for $Y(M,h)$ is
\begin{equation}
    -\frac{4(m-1)}{m-2}\Delta u+Ru=Cu^{\frac{m+2}{m-2}}\label{0-Yamabe_Lag}
\end{equation}
and the solution to the above equation gives us constant scalar curvature metric $u^{4/(n-2)}h$. $(M,h)$ is called Yamabe metric if $u\equiv 1$ mimimizes $Y(M,[h])$. When no confusion arises, we write \(Y(M)\) or \(Y(h)\) for \(Y(M,[h])\).

Through the combined efforts of H. Yamabe \cite{Yamabe}, N. Trudinger \cite{Trudinger} and T. Aubin \cite{Aubin}, it was proved $Y(M,h)\leq Y(S^m,h_{std})$, and the infimum $Y(M,h)$ can be achieved if $Y(M,h)<Y(S^m,h_{std})$ where $h_{std}$ is the standard round metric. And this final step was solved by R. Schoen \cite{Schoen}.

 When $\sigma(M)$ is attained, the critical point corresponds to a metric of constant Ricci curvature. However, $\sigma(M)$ cannot always be attained. The most famous example is from \cite{Ko85-1}\cite{Ko85-2}\cite{Schoen2}. In \cite{CGS}, it was shown using moving place method that the only solution to (\ref{0-Yamabe_Lag}) with one singularity in $\mathbb{R}^n$ is radial. Using conformal change, it was shown that on $S^{n-1}\times S^1(T)$ any solution to (\ref{0-Yamabe_Lag}) only depends on $S^1(T)$, and in particular is constant when $T$ is small.

In this paper, we are going to generalize this result to a much broader class. Let $\Lambda=\{\vec{v}_i\}_{i=1}^{n-m}$ be a lattice in $\mathbb{R}^{n-m}$ and $T(\Lambda)$ be the torus generated by it with flat metric $g_\Lambda$. Let $(I^{n-m},dx^2)$ be the unit cube with standard flat metric, then we have the natural isometry:

\begin{equation}
\begin{aligned}
            \Phi_\Lambda&\coloneqq [v_1,\cdots,v_{n-m}]:\quad I^{n-m} \rightarrow T(\Lambda) \\
\Phi_{\Lambda}^*(g_\Lambda)_{ij}&=v_i\cdot v_j\label{isometry}
\end{aligned}
\end{equation}
For simplicity, we don't distinguish $\Phi_{\Lambda}^*(g_\Lambda)$ with $g_\Lambda$. Besides, for two matrixs $A,B$, we say $A\geq B$ if $A(\vec{v},\vec{v})\geq B(\vec{v},\vec{v})$ for all vectors $\vec{v}$. Let $(N(\Lambda),g(\Lambda))=(M\times T(\Lambda),h+g_{\Lambda})$, and $E_{\Lambda},Y(\Lambda)$ be the corresponding quantities as in (\ref{0-Yamabe-func}). And let $\frac{1}{l}\Lambda$ be the lattice generated by $\{\frac{1}{l}\vec{v}_i\}_{i=1}^{n-m}$.
\begin{theorem}\label{thm_Yamabe}
    Suppose $(M^m,h)$ is Yamabe metric with $m\geq 2$ and $Y(M,[h])> 0$. Then there exists a constant $C$ depending only on $(M,h)$ so that when $g_{\Lambda}<C$, then $Y(\Lambda)$ is achieved by constant.
\end{theorem}

The motivation comes from the example \(M \times S^1(T)\). For small \(T\), the constant function \(u \equiv 1\) is a stable critical point of \(E_{g(T)}\). Therefore, it suffices to show that the minimizer \(u_T\) of \(E_{g(T)}\) satisfies \(\lim_{T \to 0} u_T = 1\).

Intuitively, as \(T \to 0\), the manifold \(M \times S^1(T)\) collapses to \(M\), and the minimization problem on \(M \times S^1(T)\) degenerates to a minimization problem on \((M,h)\). Indeed, using the isometry (\ref{isometry}), the functional (\ref{0-Yamabe-func}) can be rewritten as
\begin{equation}
\begin{aligned}
E_{g(T)}(u) 
&= \frac{\displaystyle\int_{N(1)}\! \left[c_n\frac{1}{T^{2-2/n}}(\partial_t u)^2 + T^{2/n}\bigl(c_n|\nabla_M u|^2 + R_h u^2\bigr)\right]}
      {\left(\displaystyle\int_{N(1)} u^{2n/(n-2)}\right)^{(n-2)/n}} \\
&\eqqcolon \frac{1}{T^{2-2/n}}\,A(u) \;+\; T^{2/n}\,B(u),\label{isometry-Yamabe}
\end{aligned}
\end{equation}

Taking the limit \(T \to 0\) and using the bound \(Y(g(T)) \le Y(S^n)\), we find that \(\int_{N(1)} (\partial_t u_T)^2 \to 0\); consequently, \(u_T\) converges to a limit function \(u_0\) defined on \(M\). The minimizing property of \(u_T\) forces \(u_0\) to be a minimizer of the reduced functional
\[
\tilde{E}(u)=\frac{\displaystyle\int_M\bigl(c_n|\nabla_M u|^2 + R_h u^2\bigr)}
                 {\left(\displaystyle\int_M u^{2n/(n-2)}\right)^{(n-2)/n}}.
\]

Recall that \(u \equiv 1\) minimizes the standard Yamabe functional \(E_{(M,h)}\) on \(M\). Because \(2n/(n-2) < 2m/(m-2)\), the Sobolev embedding hints that \(\tilde{E}\) might also be minimized by the constant function \(u \equiv 1\); establishing this would complete the argument.\\

Our strategy generalize direct to manifold with boundary. If $\partial N\neq \emptyset$, we can similarly define Type I and Type II Yamabe constant and calculates its Euler-Lagrangian equations\\

Type I:
\begin{equation}
\begin{aligned}
    E^I_{(N,g)}(u)=&\frac{\int_N (\frac{4(n-1)}{n-2}|\nabla_N u|^2+Ru^2)+\int_{\partial N}2H}{(\int_{N} u^{\frac{2n}{n-2}})^{\frac{n-2}{n}}}\\
    Y(N,[g])=&\inf\limits_{u\in H^1(N),u\neq 0} E^I_{(N,g)}(u)  \label{0-TypeI-func} 
\end{aligned}
\end{equation}

Type II:
\begin{equation}
\begin{aligned}
    E^{II}_{(N,g)}(u)=&\frac{\int_N (\frac{4(n-1)}{n-2}|\nabla_N u|^2+Ru^2)+\int_{\partial N}2H}{(\int_{\partial N} u^{\frac{2(n-1)}{n-2}})^{\frac{n-2}{n-1}}}\\
    Q(N,[g])=&\inf\limits_{u\in H^1(N),u\neq 0} E^{II}_{(N,g)}(u)    \label{0-TypeII-func}
\end{aligned}
\end{equation}
And their Euler-Lagrangian equations are:\\
Type I:
\begin{equation}
    \begin{aligned}
    -\frac{4(n-1)}{n-2}\Delta u+Ru &=Cu^{\frac{n+2}{n-2}} &\text{ on} \quad N\\
    \frac{\partial u}{\partial \nu}+\frac{n-2}{2(n-1)}Hu&=0 &\text{on} \quad \partial N\label{0-TypeI-Lag}
    \end{aligned}
\end{equation}
Type II:
\begin{equation}
    \begin{aligned}
    -\frac{4(n-1)}{n-2}\Delta u+Ru &=0 &\text{ on} \quad N\\
    \frac{\partial u}{\partial \nu}+\frac{n-2}{2(n-1)}Hu&=Cu^{\frac{n}{n-2}} &\text{on} \quad \partial N\label{0-TypeII-Lag}
    \end{aligned}
\end{equation}

A minimizer for the Type I Yamabe constant yields a metric with zero scalar curvature and constant mean curvature, whereas a minimizer for the Type II Yamabe constant yields a metric of constant scalar curvature and zero mean curvature. The definition of $N(\Lambda),g(\Lambda),E_{\Lambda}, Y(\Lambda),Q(\Lambda)$ generalizes in an apparent way.
\begin{theorem}\label{thm_I II}
    Suppose $h$ minimizes Type I (or II) Yamabe constant on $(M,h)$, then there exists a constant $C$ depending only on $(M,h)$ so that when $g_\Lambda<C$, then $g(\Lambda)$ minimizes Type I (Type II) Yamabe constant of $(N(\Lambda),[g(\Lambda)])$.
\end{theorem}

Notably, our strategy also applies to problems involving $Q$-curvature and to an isoperimetric-ratio type Yamabe problem from \cite{HWY08}\cite{HWY09}.

$Q$-curvature is another fundamental quantity in conformal geometry \cite{HY}.

\begin{definition}For a Riemannian manifold \((M^m,h)\), the Paneitz operator \(P\) and the $Q$-curvature are defined as

\begin{equation}
\begin{aligned}
& J=\frac{R}{2(m-1)} \quad A=\frac{1}{m-2}\left(Ric-J h\right) \\
& Q=-\Delta J-2|A|^2+\frac{m}{2} J^2 \\
& P \varphi=\Delta^2 \varphi+\operatorname{div}\left(4 A\left(\nabla \varphi, e_i\right) e_i-(m-2) J \nabla \varphi\right)+\frac{m-4}{2} Q \varphi \\
& Y_4(g)=\inf\limits_{u\in H^2(M),u\neq 0} E(u)=\inf\limits_{u\in H^2(M),u\neq 0} \frac{\int_M u P u}{\|u\|_{\frac{2 m}{m-4}}^2}\label{Q-def}
\end{aligned}
\end{equation}
\end{definition}

\begin{theorem}\label{thm-Q}
Suppose $(M^m,g)$ has constant $Ric=(m-1)g$ and $Y_4(g)$ is achieved by itself. Then there exists a constant $C$ so that $Y_4(g(\Lambda))$ is also achieved by constant for $g_{\Lambda}\leq C$ if $m\geq 6$ or $n\geq 12$.
\end{theorem}

\begin{definition}\label{definition-isoYamabe}
Consider $(M^m, \Sigma=\partial M, h)$.
Given $u$ on $\Sigma$, and let $Pu$ be its hamonic extension to $M$.
    For $p_1=\frac{2 (m-1)}{m-2}, q_1=\frac{2 m}{m-2}$, Define
    \begin{equation}
        \begin{aligned}
            E(u)&=\frac{\|P u\|_{L^{q_1}(M)}}{\|u\|_{L^{p_1}(\Sigma)}}\\
            s(M,h)&=\sup\limits_{u\in L^{p_1}(\Sigma), u\neq 0} E(u)
        \end{aligned}
    \end{equation}
\end{definition}

It was shown in \cite{HWY08}\cite{HWY09} that
\begin{theorem}\label{HWY}
For \(1 \leq p < \infty\) and \(1 \leq q \leq \frac{m p}{m-1}\), the operator \(P: L^p(\Sigma) \to L^q(M)\) is bounded, and compact when \(q < \frac{m p}{m-1}\). Hence \(s(M,h)\) is well-defined.

When $(M,h)$ is the standard disk, the constant \(s(B^m, dx^2)\) is attained by the function \(u_\phi\) induced by a Möbius transformation \(\phi\) (so \(\phi^*(dx^2)=u_\phi^{4m/(m-2)}dx^2\)). If \(s(M,h) < S(B^m, dx^2)\), then \(s(M,h)\) is also attained.
\end{theorem}

And for this Yamabe type quantity, we also have a similar result.

\begin{theorem} \label{thm-isoratio}
     Let $N(\Lambda)=T^{n-m}(\Lambda) \times B^m,\Sigma(\Lambda)=S^{m-1}\times T^{n-m}(\Lambda), g(\Lambda)=dx^2+g_{\Lambda}$, and $E_{\Lambda}$, $s(\Lambda)$ be the related functional in Definition \ref{definition-isoYamabe}. For arbitrary $\Lambda$, there exist a $k(\Lambda)$ so that if $l>k(\Lambda)$, then $s(\Lambda/2^l)$ is achieved by const if $m>3$ or $n>6$.
\end{theorem}

This paper is organized as follows. Section 2 presents the main ideas and provides a detailed proof of \textbf{Theorem \ref{thm_Yamabe}}. (the similar proof of \textbf{Theorem \ref{thm_I II}} is given in one sentence). Applying the same strategy, Sections 3 and 4 prove \textbf{Theorem \ref{thm-Q}} and \textbf{Theorem \ref{thm-isoratio}}, highlighting only the necessary modifications. Finally, Section 5 states a conjecture on the Yamabe constant of \(M \times T^{n-m}\) and connects it to the Yamabe invariant of \(S^2 \times S^2\).

\textbf{Acknowledgement} The authors are grateful to Prof. Kazuo Akutagawa, Mijia Lai, Jimmy Petean and Xiaodong Wang for their helpful discussion. Fang Wang’s research was supported by National Natural Science Foundation of China No. 12271348 and Zhixin Wang's research was supported by China Postdoctoral Science Fundation 2025T180838.

\section{Yamabe Constant}

This section provides a detailed proof of \textbf{Theorem \ref{thm_Yamabe}}, divided into three parts.

1. Stability: We estimate the eigenvalues of \((N(\Lambda),g(\Lambda))\) and show that when \(\Lambda\) is sufficiently small, \(u \equiv 1\) is a stable critical point for \(E(\Lambda)\). More precisely, within an \(\epsilon\)-neighborhood of the constant function, \(u \equiv 1\) uniquely minimizes the energy, up to constants.

2. Limit of minimizers: Fix a small \(\Lambda_0 = \{\vec{a}_i\}_{i=1}^{n-m}\) and set 
 \begin{equation}
 \Lambda_k=\frac{1}{2^k}\Lambda_0\coloneqq \{\frac{1}{2^k}\vec{a}_i\}_{i=0}^{n-m}\notag
\end{equation}
 Let \(v_k\) be a minimizer for \(Y(\Lambda_k)\) and lift it to a function \(u_k\) on \(N(\Lambda_0)\). We prove that \(u_k\) converges to a limit function \(u_0\). Because \(u_0\) inherits arbitrarily small periodicity along \(T^{n-m}\), it must be independent of the torus coordinates and can therefore be regarded as a function on \(M\). Moreover, the minimizing property of each \(v_k\) is preserved in the limit, which implies its energy does not exceed that of $u\equiv 1$, which is (\ref{lemma-3-convergence-2}).

3. Identification of the limit: Using the ``small energy" property (\ref{lemma-3-convergence-2}) and the assmuption that $(M,h)$ is Yamabe metric, we conclude that \(u_0 \equiv 1\). This convergence, together with the stability established in part 1, implies that \(Y(g(\Lambda_k))\) is minimized by a constant for all sufficiently large \(k\).

Finally, we briefly adapt the argument to prove \textbf{Theorem \ref{thm_I II}}.

\subsection{Stability}\,

The motivation of this section is the following computation. On a general Riemannian manifold $(M,g)$ with constant scalar curvature $R_h$. Suppose $\int_{M}u=0$, consider $E_{(M^m,h)}(1+tu)$ and computes its second derivative, we have
\begin{equation}
    \frac{\partial^2 }{\partial t^2}E_{(M,h)}(1+tu)=\frac{8}{(m-2)|M|}\big(\int_{M}(m-1)|\nabla_M u|^2-R_hu^2\big)\label{2-2nd derivative}
\end{equation}
In particular, in the setting of \textbf{Theorem \ref{thm_Yamabe}}, we have
\begin{lemma}\label{lemma-eigen-M}
    Let $\lambda(M),\lambda(\Lambda)$ be the first eigenvalue of Laplacian for $(M,h)$ and $(N(\Lambda),g(\Lambda))$ respectively. Suppose $(M^m,h)$ is a Yamabe metric, then we have
    \begin{equation}
        (m-1)\lambda(M)\geq R_h\label{eigen-M}
    \end{equation}
    On $(N(\Lambda),g(\Lambda))$, $u\equiv 1$ is a stable critical point for $E_{\Lambda}$ provided that
    \begin{equation}
        (n-1)\lambda(\Lambda)> R_{g(\Lambda)}=R_h\label{eigen-M}
    \end{equation}
\end{lemma}

We are going to show that $\lambda(\Lambda)=\lambda(M)$ for $\Lambda$ small enough, and the stability of $u\equiv 1$ for $E_{\Lambda}$ follows directly from the lemma above.

\begin{lemma}\label{lemma_eigen}
    Let $\lambda(\Lambda)$ be the 1st non-zero eigenvalue for Laplacian of $(N(\Lambda),g(\Lambda))$, and $\lambda(M)$ be the 1st non-zero eigenvalue of Laplacian for $(M,h)$. Then there exists a $C$ depending on $M$ only so that
\begin{equation}
    \lambda(\Lambda)=\lambda(M) \text{ for }g_\Lambda\leq C
\end{equation}
\end{lemma}
\begin{proof}

We splits $\int_{N(\Lambda)}|\nabla_{N(\Lambda)} u|^2$ into
\begin{equation}
    \int_{N(\Lambda)}|\nabla_{N(\Lambda)} u|^2dv_\Lambda=\int_{N(\Lambda)}|\nabla_M u|^2+\int_{N(\Lambda)}|\nabla_Tu|^2
\end{equation}
where $\nabla_M u$ is the push-forward of $\nabla_M u$ via the projection $N(\Lambda)\rightarrow M$, and $\nabla_T u$ is the push-forward of $\nabla_M u$ via the projection $N(\Lambda)\rightarrow T^{n-m}$. Using isometry (\ref{isometry}), $\lambda(\Lambda)$ can be rewritten as
\begin{equation}
 \begin{aligned}       \lambda(\Lambda)&=\inf\frac{\int_{T(\Lambda)}\int_{M}(|\nabla_Tu|^2+|\nabla_M u|^2)dv_Mdv_{T(\Lambda)} }{\int_{T(\Lambda)}\int_{M}u^2dv_Mdv_{T(\Lambda)}}\\
        &=\inf\frac{\int_{T(\vec{1})}\int_{M}(g_{\Lambda}^{ij}\frac{\partial u}{\partial x^i}\frac{\partial u}{\partial x^j}+|\nabla_M u|^2)dv_Mdv_{T(\vec{1})}}{\int_{T(\vec{1})}\int_{M}u^2dv_Mdv_{T(\vec{1})}} \label{2-eigenvalue}
\end{aligned}
\end{equation}
where $u\in H^1$ and $u$ is not constant. Therefore, $\lambda(\Lambda_1)\geq \lambda(\Lambda_2)$ provided that $g_{\Lambda_1}\leq g_{\Lambda_2}$.

Next, consider the sequence $(N(\Lambda_p), g(\Lambda_p))$ where $\Lambda_p$ is the cube with side length $\frac{1}{2^p}$ for $p \in \mathbb{Z}_+$, and let $u_p$ denote its first eigenfunction. The apparent $p^{n-m}$-sheeted covering map $I^{n-m}=\Lambda_1 \to T^{n-m}(\Lambda_p)$ allows us to pull back $u_p$ to an eigenfunction $\tilde{u}_p$ on $(N(\Lambda_1), g(\Lambda_1))$ with eigenvalue $\lambda(\Lambda_p)$.

Using the first eigenfunction $u_M$ and eigenvalue $\lambda_M$ of $(M, h)$ as a test function, we find that $\lambda(\Lambda_p) \leq \lambda(M)$ for each $p$. $g_{\Lambda_p}<g_{\Lambda_q}$ for $p>q$, and combined with the monotonicity established earlier, this implies the pulled-back eigenvalues $\lambda(\Lambda_p)$ form an increasing sequence in $p$ that is bounded above by $\lambda(M)$.

Since the number of eigenvalues of $(N(\Lambda_1, g((\Lambda_1))$ below $\lambda(M)$ is finite, the sequence $\lambda((\Lambda_p)$ must become constant for all sufficiently large $p$. From equation (\ref{2-eigenvalue}), this constancy implies that the corresponding functions $\bar{u}_p$ are independent of the $T^{n-m}$-direction. Consequently, $\lambda(\Lambda_p) = \lambda(M)$ for all large $p$'s. Fix such a $p$. For $g_\Lambda\leq g_{\Lambda_{p}}$, we have $\lambda(M)\geq \lambda(\Lambda)\geq \lambda(\Lambda_{p})=\lambda(M)$, and we are done.
\end{proof}

\begin{lemma}\label{lemma-stability}
    Let $(M,h)$ be as in \textbf{Theorem \ref{thm_Yamabe}}. Fix a small $\Lambda_0$ from \textbf{Lemma \ref{lemma_eigen}} so that $\lambda(\Lambda)=\lambda(M)$. Then there exists $\epsilon_0$ so that for any $t \leq \epsilon_0$, $\|u\|_{H^1(N(\Lambda))}=1$ and $\int_{N(\Lambda)}u=0$, we have
    \begin{equation}
        E_{\Lambda_0}(1+tu)\geq E_{\Lambda_0}(1)
    \end{equation}
    with equality iff $t=0$.
\end{lemma}

\begin{proof}
    For simplicity, let $q=\frac{2n}{n-2}$ and drop the index $\Lambda_0$ since it's fixed. Also let $\|u\|_{H_1}^2=\int_N \frac{4(n-1)}{n-2}|\nabla_N u|^2+R_hu^2$ which is equivalent to the usual one.
    \begin{equation}
        E(1+tu)=  \frac{\int_N \frac{4(n-1)}{n-2}t^2|\nabla_N u|^2+R_h(1+t^2u^2)}{(\int_{N} |1+tu|^q)^{2/q}}=\frac{f(t)}{g(t)}
    \end{equation}
    Then we have
    \begin{equation}
    \begin{aligned}
        f&=t^2+R_h|N|\\
        \frac{g'}{g}&=\frac{2}{\int_{N}|1+tu|^q}\int_{N}u|1+tu|^{q-1}\\
        (\frac{g'}{g})'&=2(q-1)\frac{\int_{N}u^2|1+tu|^{q-2}}{\int_{N}|1+tu|^q}-2q\big(\frac{\int_{N}u|1+tu|^{q-1}}{\int_{N}|1+tu|^q}\big)^2
     \end{aligned}   
    \end{equation}

    And direct computation shows that
    \begin{equation}
    \begin{aligned}
        \frac{\partial^2}{\partial t^2}E(1+tu)&=\frac{1}{g}(f''-f(\frac{g'}{g})'+f(\frac{g'}{g})^2-2\frac{g'}{g}f')\\
        &=\frac{2}{g}\bigg(1-(t^2+R_h|N|)(q-1)\frac{\int_{N}u^2|1+tu|^{q-2}}{\int_{N}|1+tu|^q}\\
        &\quad+\frac{(t^2+R_h|N|)(q+2)}{4}(\frac{g'}{g})^2-2t\frac{g'}{g}\bigg)
    \end{aligned}
    \end{equation}
    View $\frac{(t^2+R_h|N|)(q+2)}{4}(\frac{g'}{g})^2-2t\frac{g'}{g}$ as quadratic in $\frac{g'}{g}$ and it's bounded below by $-Ct^2$. By Sobolev inequality and $\|u\|_{H^1}=1$, $\int_{N}|1+tu|^1\geq |N|-C(N)t$. So
    \begin{equation}
        \frac{\partial^2}{\partial t^2}E(1+tu)
        \geq \frac{2}{g}\bigg(1-\frac{(t^2+R_h|N|)(q-1)}{|N|-C(N)t}\int_{N}u^2|1+tu|^{q-2}-C(N)t^2\bigg)\label{2-2}
    \end{equation}
    It remains to estimate $\int_{N}u^2|1+tu|^{q-2}$. Since $q-2>0$, for an arbitrary small $\delta$ we have that
    \begin{equation}
        |1+x|^{q-2}\leq |1+|x||^{q-2}\leq 1+\delta+C(n,\delta)|x|^{q-2}
    \end{equation}
  Thus applying Sobolev trace inequality
  \begin{equation}
      \int_{N}u^2|1+tu|^{q-2}\leq \int_N C(n,\delta)t^{q-2}u^q+(1+\delta)u^2\leq C(\delta,N)t^{q-2}+\int_N (1+\delta)u^2.
  \end{equation}
  Now (\ref{2-2}) becomes
  \begin{equation}
        \frac{\partial^2}{\partial t^2}E(1+tu)
        \geq \frac{2}{g}\bigg(1-\frac{(t^2+R_h|N|)(q-1)}{|N|-C(N)t}(1+\delta)\int_{N}u^2-C(N,\delta)t^{q-2}-C(N,\delta)t^2\bigg)\label{2-middle}
    \end{equation}
If $t=\delta=0$, using $\|u\|_{H^1}=\int_N \frac{4(n-1)}{n-2}|\nabla_N u|^2+R_hu^2=1$, we have

\begin{equation}
    \text{RHS of }(\ref{2-middle})\geq \frac{8}{(n-2)g}\big(\int_{N}(n-1)|\nabla_Nu|^2-R_hu^2\big)> 0
\end{equation}
We used \textbf{Lemma \ref{lemma-eigen-M},\ref{lemma_eigen}} in the above inequality. This implies that if we pick $\delta,t$ small enough, then $\frac{\partial^2}{\partial t^2}E(1+tu)>0$.
\end{proof}

\begin{remark}
    This theorem will be proved by showing that $\partial_t^2 E(1+tu)\geq 0$ for $t\leq \epsilon$. I tried to prove this result by using the weak compactness of $\{\|u\|_{H_1=1},\int u=0\}$, but $u$ might converge to 0. I think there should be a simpler way to show stability from (\ref{2-2nd derivative}).
\end{remark}

\subsection{Subconvergence to $u_0$ of small energy}\,


The convergence argument proceeds as follows. By the Concentration-Compactness Theorem, the sequence of measures $u_i^{2n/(n-2)} dv$ converges to $u^{2n/(n-2)} dv + \sum \nu^j \delta_j$ in the sense of measures. The condition of infinitely small periodicity forces the vanishing of all Dirac masses, implying that $u_i \to u$ in $L^{2n/(n-2)}$. Furthermore, this periodicity implies that the limit $u$ is independent of the $T^{n-m}$ coordinates and can therefore be regarded as a function on $M$. The minimization property of the $v_i$, combined with $L^{2n/(n-2)}$ convergence, implies that the energy of $u$ is at most that of the constant function.

In the next subsection, we will show that using the smallness of $E_{\Lambda_0}$, Lemma \ref{3-s} implies that $u_0$ must itself be constant, specifically $u \equiv 1$.

To do this we need the following concentration-compactness theorem (see \cite{SM} for example):

\begin{theorem}(\textbf{Concentration-Compactness}) . Let $(N,g)$ be a closed Riemannian manifold. Suppose $u_k \rightharpoonup u$ weakly in $H^1(N)$ and $\mu_k = |\nabla^k u_k|^2 dv \rightharpoonup \mu$, $\nu_k = |u_k|^{\frac{2n}{n-2}} dx \rightharpoonup \nu$ weakly in the sense of measures where $\mu$ and $\nu$ are bounded non-negative measures on $N$.
Then we have:

(I°) There exists some at most countable set $J$, a family $\{x^{(j)}; j \in J\}$ of distinct points in $N$, and a family $\{\nu^{(j)}; j \in J\}$ of positive numbers such that
\begin{equation}\nu = |u|^{\frac{2n}{n-2}} dv + \sum_{j \in J} \nu^{(j)} \delta_{x^{(j)}},\label{3-concentration-1}
\end{equation}
where $\delta_x$ is the Dirac-mass of mass 1 concentrated at $x \in N$.

(II°) In addition we have
\begin{equation}
\mu \geq |\nabla u|^2 dv + \sum_{j \in J} \mu^{(j)} \delta_{x^{(j)}}\label{3-concentration-2}
\end{equation}
for some family $\{\mu^{(j)}; j \in J\}$, $\mu^{(j)} > 0$ satisfying
$$S\left(\nu^{(j)}\right)^{(n-2)/n} \leq \mu^{(j)}, \quad \text{for all } j \in J.$$

In particular, $\sum_{j \in J} \left(\nu^{(j)}\right)^{(n-2)/n} < \infty$.
\end{theorem}

\begin{lemma}\label{lemma-3-convergence}
    Let \(\Lambda_k = 2^{-k} \Lambda_0\) and let \(v_k\) be a minimizer for \(Y(\Lambda_k)\). Lift \(v_k\) to a function \(u_k\) on \(N(\Lambda_0)\) via the \(k^{n-m}\)-sheeted covering map, and normalize it so that
\[
\int_{N(\Lambda_0)} u_k^{2n/(n-2)} = \operatorname{Vol}(N(\Lambda_0)).
\]
Then
\[
u_k \rightharpoonup u_0 \quad \text{weakly in } H^1(N(\Lambda_0)).
\]
Furthermore, \(u_0\) is independent of the \(T(\Lambda_0)\)-factor and may therefore be regarded as a function on \(M\). This limit function satisfies the inequality on $M$
    \begin{equation}
            \frac{\int_M|\nabla u_0|^2+\frac{n-2}{4(n-1)}R_h u_0^2dv_g}{(\int_M u_0^{2n/(n-2)})^{(n-2)/n}dv_g}\leq \frac{n-2}{4(n-1)}R_hV(M)^{2/n}\label{lemma-3-convergence-2}
    \end{equation}
where we have identified the metric on \(M\) with its pullback under the projection.
\end{lemma}
\begin{proof}
 Since each \(v_k\) is a minimizer,
\begin{equation}
    \frac{E_{\Lambda_0}(u_k)}{V(N(\Lambda_0))^{2/n}}=\frac{E_{\Lambda_k}(v_k)}{V(N(\Lambda_0))^{2/n}}\leq \frac{E_{\Lambda_k}(1)}{V(N(\Lambda_k))^{2/n}}=\frac{E_{\Lambda_0}(1)}{V(N(\Lambda_0))^{2/n}}\label{3-convergence-middle1}
\end{equation}

Recall that the functions \(u_k\) have been normalized so that $\int_{N(\Lambda_0)} u_k^{2n/(n-2)} = \operatorname{Vol}(N(\Lambda_0)).$
Hence, we may extract a subsequence such that \(u_k \rightharpoonup u_0\) weakly in \(H^1\), and the associated concentration–compactness measures \(\nu, \mu\) satisfy the usual decomposition. Because \(u_k\) is periodic with period \(\vec{a}_i/2^l\) in each direction of \(T^{n-m}\) for all \(k \ge l\), the same periodicity is inherited by the limits \(\nu, \mu, u_0\) for every \(l \in \mathbb{Z}_+\). This infinite divisibility of the period, together with the fact that \(\nu(N(\Lambda_0)) = \operatorname{Vol}(N(\Lambda_0)) < \infty\), forces the set of “concentration points’’ in (\ref{3-concentration-1}) to be empty, i.e. \(J = \emptyset\). Consequently,
\[
\|u_0\|_{L^{2n/(n-2)}} = \operatorname{Vol}(N(\Lambda_0)) = \lim_{k\to\infty} \|u_k\|_{L^{2n/(n-2)}} ,
\]
and by \cite{BE}, we obtain strong convergence
\[
u_k \to u_0 \quad\text{in } L^{2n/(n-2)} .
\]

Finally, combining this strong convergence with (\ref{3-concentration-2}) and passing to the limit in (\ref{3-convergence-middle1}) gives
\begin{equation}
E_{\Lambda_0}(u_0) \le E_{\Lambda_0}(\text{const}) .
\label{3-convergence-middle2}
\end{equation}

We now show that $u_0$ must be independent of the $T^{n-m}$ factor. Let $\theta_i$ ($i=1,\dots,n-m$) denote coordinates on the torus. For each $i$, let $\phi_\epsilon(\theta_i)$ be a mollifier in the $\theta_i$ variable, and define
\begin{equation}
u_{\epsilon}(x, \theta_1, \dots, \theta_{n-m}) = \int_{S^1} u_0(x, \theta_1, \dots, \theta_i-\xi, \dots, \theta_{n-m}) \phi_{\epsilon}(\xi)  d\xi, \quad x \in M.
\end{equation}
The function $u_\epsilon$ is smooth in $\theta_i$ and inherits the periodicity $\Lambda_0/2^k$ for every $k$, implying that $\partial u_\epsilon / \partial \theta_i = 0$ for every $\epsilon > 0$. Taking the limit $\epsilon \to 0$, we find $u_\epsilon \to u_0$ in $H^1(N(\Lambda_0))$, whence $\partial u_0 / \partial \theta_i = 0$. Therefore, $u_0$ is independent of the $T^{n-m}$ coordinates.

Now we can view $u_0$ as a function on $M$, and (\ref{lemma-3-convergence-2}) follows from (\ref{3-convergence-middle2}).

\end{proof}

\subsection{$u_0$ is constant}\,

We start with a lemma.
\begin{lemma}\label{3-s}
    \label{s_decrease}
    Let $(M,g)$ be a closed Riemannian manifold with $Vol(M)=1$ and define
    \begin{equation}
        \begin{aligned}
            F_{\lambda,q}(u)&=\frac{e(u)+\int_M \lambda u^2dv_g}{(\int_Mu^q)^{2/q}dv_g}\\
            s(\lambda,q)&=\inf_{u\in H^1,u\neq 0}F_{\lambda,q}(u)\label{3-lemmaODE-def}
        \end{aligned}
    \end{equation}
    where $\lambda>0$, $e(u)=\int_M|\nabla u|^2$ and $ 2\leq q\leq \frac{2m}{m-2}$. Suppose $s(\lambda_0,q_0)$ is achieved by constant, then $s(\lambda_1,q_1)$ is also achieved by constants for
    \begin{equation}
        \frac{q_1}{(q_1-2)\lambda_1}\geq\frac{q_0}{(q_0-2)\lambda_0},\quad q_1<q_0\label{s_decrease_condition}
    \end{equation}

\end{lemma}

\begin{proof}

We argue by contradiction. Because $\text{Vol}(M)=1$, a constant function fails to be a minimizer for $s(\lambda,q)$ if and only if $s(\lambda,q) - \lambda < 0$.

\textbf{Claim:} Starting from arbitrary $(\lambda_1,q_1)$ with $q_1>2,\lambda_1>0$, the quantity $s(\lambda,q) - \lambda$ is non-increasing along the curve
\begin{equation}
    \frac{q}{q-2} = C\lambda \label{3_ODE_solution}
\end{equation}
in the direction of increasing $q$, and is strict decreasing at points where $s(\lambda,q)$ admits a non-constant minimizer. Here the constant $C = \frac{q_1+1}{(q_1-1)\lambda_1 }$ is chosen so that $(\lambda_1,q_1)$ lies on the curve.

If the claim holds, then the existence of a non‑constant minimizer for \(s(\lambda_1,q_1)\) implies \(s(\lambda,q) < \lambda\) for every point \((\lambda,q)\) on the curve (\ref{3_ODE_solution}) with \(q > q_1\). Now choose \(\lambda_2\) so that \((\lambda_2,q_0)\) lies on the same curve, (so $s(\lambda_2,q_0)<\lambda_2$) and condition (\ref{s_decrease_condition}) gives \(\lambda_0 \ge \lambda_2\).For each fixed \(u\), the functional \(F_{\lambda,q}(u)\) is linear in \(\lambda\); taking the infimum therefore shows that \(s(\lambda,q)\) is concave in \(\lambda\). Therefore \(s(\lambda_2,q_0) < \lambda_2\) forces \(s(\lambda_0,q_0) < \lambda_0\), which says $s(\lambda_0,q_0)$ is not minimizes bo constant, contradicting the original assumption.

     It remains to prove the claim. Along the curve (\ref{3_ODE_solution}), we can view $s,q$ as functions of $\lambda$ only, i.e. $s(\lambda),q(\lambda)$. Let $u$ be minimizer (if it admits both constant and non-constant minimizers, we pick the non-constant one) for $s(\lambda_1,q_1)$ with normalization $\int u^{q_1}=1$, and we have $s(\lambda,q(\lambda))\leq E_{\lambda,q(\lambda)}(u)$ with equality at $(\lambda_1,q_1)$. And we plan to prove the following:\\
\begin{equation}
    \frac{\partial}{\partial (-\lambda)}(E_{\lambda,q(\lambda)}(u)-\lambda)\big|_{\lambda_1}\leq 0\label{s_middle0}
\end{equation}
with equality iff $u\equiv 1$. One calculates that\\
\begin{equation}
    \frac{\partial}{\partial q}\left(\int_{M} u^{q}\right)^{\frac{2}{q}}=(\int_{M} u^{q})^{\frac{2}{q}}[-\frac{2}{(q)^2}\log \int_{M} u^{q}+\frac{2}{q}\frac{\int_{M} u^{q}\log u}{\int_{M} u^{q}}]\notag
\end{equation}
We assumed that $\int_{M}u^{q_1}=Vol(M)=1$, so $\frac{\partial }{\partial q}(\int_Mu^q)^{2/q}|_{q_1}=\frac{2}{q}\int_Mu^{q_1}\log u$, and consequently\\
\begin{equation}
    \frac{\partial}{\partial \lambda}E_{\lambda,q(\lambda)}(u)\big|_{\lambda_1}
    =\int_{M} u^2-\frac{2s(\lambda_1)}{q_1}q'(\lambda_1)\int_{M} u^{q_1}\log u\notag
\end{equation}
So it suffices to show\\
\begin{equation}
    \int_{M} u^2-\frac{2s(\lambda_1)}{q_1}q'(\lambda_1)\int_{M} u^{q_1}\log u\geq 1\label{s_derivative}
\end{equation}
with equality iff $u\equiv 1$. By H\"older inequality and $Vol(M)=\int_{M} u^{q_1}=1$, we have $\int_{M} u^{q_1+\delta}\geq 1$ for $\delta>0$ and it follows that $\int_{M} u^{q_1}\log u\geq 0$. Also note that $\int_{M} u^x$ is a strict convex function in $x$ for $x>0$, we have\\
\begin{equation}
    \begin{aligned}
        &\int_Mu^{q_1}-\int_Mu^2\leq (q_1-2)\int_M u^{q_1}\log u\\ 
    \Longrightarrow &\int_{M} u^2+(q_1-2)\int_{M} u^{q_1}\log u\geq \int_{M} u^{q_1}=1 \label{s_derivative_2}
    \end{aligned}
\end{equation}
Since $s(\lambda_1,q_1)\leq \lambda_1$, (\ref{s_derivative}) follows from (\ref{s_derivative_2}) provided\\
\begin{equation}
    \frac{-2q'(\lambda)\lambda}{q}\geq q-2\notag
\end{equation}
The solution for the equality is exactly (\ref{3_ODE_solution}). This computation is carried out at $(\lambda_1,q_1)$, but this works at all points, and we are done.
\end{proof}
\begin{remark} 

The lemma is motivated by the following observation. Since \(L^{q_1} \hookrightarrow L^{q_0}\) for \(q_1 > q_0\), controlling \(s(\lambda, q)\) becomes more difficult as \(q\) increases. This challenge is particularly acute at the critical exponent \(q = \frac{2n}{n-2}\) due to the loss of compactness. Therefore, if one has good estimates for \(s(\lambda_0, q_0)\), it is reasonable to expect they can be transferred to obtain estimates for \(s(\lambda_1, q_1)\) when \(q_1 < q_0\) for a suitable \(\lambda_1\).

From this viewpoint, the condition \(q_1 \leq q_0\) in (\ref{s_decrease_condition}) is essential, which is also reflected in the proof. In deriving (\ref{s_middle0}), the derivative was taken in the direction of decreasing \(\lambda\) (equivalently, increasing \(q\)). Taking the derivative in the opposite direction would yield a lower bound for \(E_{\lambda,q}\), which is not useful as we only possess the upper bound \(s(\lambda, q) \leq E_{\lambda,q}(u)\).

Finally, we note that the inequality (\ref{s_decrease_condition}) is not sharp, at least for spheres. For instance, \cite{BF} shows that on \(S^n\), the function \(s(\lambda,q)\) is minimized by constants provided that \(\lambda(q-2) \leq n\).
\end{remark}

\begin{remark}\label{choice-energy}
We note that the expression \(e(u)\) will differ in later sections: for the \(Q\)-curvature problem the exponent satisfies \(2 \leq q \leq 2m/(m-4)\), and for the type II Yamabe problem it takes the form \(e(u)\int_X |\nabla_X u|^2\) where \(M = \partial X\). In all cases, however, the proof proceeds identically and can be presented in a unified way.
\end{remark}

With the help of the previous lemma, we can shows that $u_0=const$.
\begin{lemma}\label{u0-const}
    $u_0$ from \textbf{Lemma \ref{lemma-3-convergence}} is const.
\end{lemma}
\begin{proof}
We will apply \textbf{Lemma \ref{3-s}}. For this purpose, we scale the metric so that $\text{Vol}(M, \bar{h}) = 1$, where $\bar{h}$ is a constant multiple of $h$. Let $\bar{F}_{\lambda,q}$ and $\bar{s}(\lambda,q)$ denote the quantities defined in (\ref{3-lemmaODE-def}) for the scaled metric $\bar{h}$. A direct computation shows that
\begin{equation}
F_{\lambda,q}(u) = \text{Vol}(M, h)^{\frac{n-2}{n} - \frac{2}{q}} \, \bar{F}_{\text{Vol}(M,h)^{2/n}\lambda,q}(u).
\end{equation}
Consequently, $s(\lambda,q)$ is achieved by constant functions if and only if $\bar{s}(\text{Vol}(M,h)^{2/n}\lambda,q)$ is. Since multiplying $\lambda$ by $\text{Vol}(M,h0)^{2/n}$ does not affect condition (\ref{s_decrease_condition}), the lemma remains valid for the original manifold $(M,h)$.

Since we assume \(Y(M,h)\) is achieved by a constant function, it follows that \(F_{\lambda_0,q_0}\) is achieved by constants for the parameters
\[
(\lambda_0,q_0)=\left(\frac{m-2}{4(m-1)}R_h, \frac{2m}{m-2}\right).
\]

Now consider the parameters
\[
(\lambda_1,q_1)=\left(\frac{n-2}{4(n-1)}R_h, \frac{2n}{n-2}\right).
\]
A direct computation shows that for $n>m\geq 3$
\begin{equation}
    \frac{q_1}{(q_1-2)\lambda_1}-\frac{q_0}{(q_0-2)\lambda_0}=\frac{2(n-m)(mn - 2m - 2n + 2)}{(m - 2)(n - 2)}\geq 0\notag
\end{equation}
so condition (\ref{s_decrease_condition}) holds. Consequently, \(u \equiv 1\) is the unique minimizer of \(F_{\lambda_1,q_1}\).

Moreover, by (\ref{lemma-3-convergence-2}), we have \(F_{\lambda_1,q_1}(u) \leq F_{\lambda_1,q_1}(\text{const})\). This forces \(u_0 \equiv 1\) and implies that the inequality in (\ref{3-concentration-2}) is in fact an equality. Therefore, \(u_k \to 1\) in \(H^1(N(\Lambda_0))\).
\end{proof}



With there preparation lemmas, \textbf{Theorem \ref{thm_Yamabe}} follows easily.

\noindent\textit{Proof of \textbf{Theorem \ref{thm_Yamabe}}}:

Since a subsequence of \(u_k\) (still denoted by \(u_k\) for simplicity) converges to \(u_0 \equiv 1\), it must eventually lie within the \(\epsilon_0\)-neighborhood described in \textbf{Lemma \ref{lemma-stability}}. Together with (\ref{3-convergence-middle1}), this implies that \(u_k \equiv 1\) for all sufficiently large \(k\).

So far we have proved that for arbitrary $\Lambda$, there exist infinitely many $k$ such that $Y(\frac{1}{2^k}\Lambda)$ is achieved by a constant metric. Fix one such $k$ and set $\Lambda_1 = \frac{1}{2^k}\Lambda$; then $Y(\Lambda_1)$ is achieved by a constant metric. To complete the proof, we need to show that for any $g_{\Lambda_2}$ with $g_{\Lambda_2} < g_{\Lambda_1}$, $Y(\Lambda_2)$ is also achieved by a constant metric.

Using isometry (\ref{isometry}), $E_{\Lambda}(u)$ can be rewriten as
\begin{equation}
    \begin{aligned}
       \frac{E_{\Lambda}(u)}{V(T(\Lambda))^{2/n}}=& \frac{\int_{N(\Lambda)}\frac{4(n-1)}{n-2} (|\nabla_M u|^2+|\nabla_Tu^2|)+R_hu^2dv_{N(\Lambda)}}{(\int_{N(\Lambda)} u^{\frac{2n}{n-2}})^{\frac{n-2}{n}}V(T(\Lambda))^{2/n}}\\
        =& \frac{\int_M\int_{I^{n-m}}\frac{4(n-1)}{n-2}(
        g_\Lambda^{ij}u_iu_j+|\nabla_M u|^2)+R_hu^2dxdv_M}{(\int_M\int_{I^{n-m}}u^{\frac{2n}{n-2}})^{(n-2)/n}}\label{3-E-comparison}
    \end{aligned}
\end{equation}
where $\{x_i\}$ are coordinates for $I^{n-m}$ and $u_i=\frac{\partial u}{\partial x_i}$. This implies
\[\frac{E_{\Lambda_2}(u)}{V(T(\Lambda_2))^{2/n}}\geq \frac{E_{\Lambda_1}(u)}{V(T(\Lambda_1))^{2/n}}\]
if $g_{\Lambda_2}< g_{\Lambda_1}$, thus $Y(\Lambda_2)$ is also achieved by constant if $Y(\Lambda_1)$ is for any $g_{\Lambda_2}\leq g_{\Lambda_1}$. Furthermore, $u=1$ is the only minimizer for $Y(\Lambda_2)$ if $g_{\Lambda_2}< g_{\Lambda_1}$.

\noindent\textit{Proof of Theorem \ref{thm_I II}}:

The proof for Type I Yamabe problem is exactly the same. For Type II Yamabe problem direct computation shows that stability is related to Steklov eigenvalue, combine with Remark \ref{choice-energy} and the proof generalize directly.
\qed

\section{Q-curvature}

Let $N(\Lambda),g(\Lambda)$ defined as before, and next we will prove \textbf{Theorem \ref{thm-Q}}

\begin{proof}
    During this proof, quantities with subindex $M$, $N$, $T$ denotes those for $M$ $N(\Lambda)$ and $T^{n-m}$ respectively.

    Direct computation show that on $(M,h)$:
 $$
\begin{aligned}
    &J_M=\frac{m}{2} \quad A_M=\frac{1}{m-2}\left((m-1)-\frac{m}{2}\right) h=\frac{h}{2}\\
Q_M & =-2 m \cdot\left(\frac{1}{2}\right)^2+\frac{m^3}{8}=\frac{m}{8}\left(m^2-4\right) \\
P_M \varphi & =\Delta_M^2 \varphi+\operatorname{div}_M\left(2 \nabla_M \varphi-\frac{m(4-2)}{2} \nabla_M \varphi\right)+\frac{1}{16} m\left(m^2-4\right)(m-4) \varphi \\
& =\Delta_M^2 \varphi+\operatorname{div}_M\left(-\left(\frac{m^2-2 m-4}{2}\right) \nabla_M \varphi\right)+\frac{1}{16} m\left(m^2-4\right)(m-4) \varphi \\
\int_M P_M \varphi \cdot \varphi & =\int_M(\Delta_M \varphi)^2+\frac{m^2-2 m-4}{2}|\nabla_M \varphi|^2+\frac{1}{16} m\left(m^2-4\right)(m-4) \varphi^2
\end{aligned}
$$
and on $N(\Lambda)$

 $\begin{aligned}J_N&=\frac{m(m-1)}{2(n-1)} \\
    A_N&=\frac{1}{n-2}\left(\begin{array}{cc}{\left[m-1-\frac{m(m-1)}{2(m-1)}\right] h} & 0 \\ 0 & -\frac{m(m-1)}{2(m-1)} g_T\end{array}\right).\\
    &=\left(\begin{array}{cc}\frac{(m-1)(2 n-m-2)}{2(m-1)(n-2)} h & 0 \\ 0 & \frac{-m(m-1)}{2(n-1)(n-2)} g_T\end{array}\right) 
\end{aligned}$

   $  \begin{aligned} Q_N&=-2|A_N|^2+\frac{n}{2} J_N^2 \\ & =\frac{m(m-1)^2}{8(n-1)^2(n-2)^2}\left(m n^3-4 m n^2-16 n^2+16 m n+32 n-16 m-16\right)\end{aligned}$

For $i=1, \cdots, m,$

$$
\begin{aligned}
    & \quad 4 A_N\left(\nabla \varphi, e_i\right)-(n-2)J_N \varphi_i\\
& =\left[4 \cdot \frac{(m-1)(2 n-m-2)}{2(n-1)(m-2)}-\frac{m(m-1)(n-2)}{2(m-1)}\right] \varphi_i \\
& =\frac{m-1}{2(n-1)(n-2)}\left(-m n^2+4 m n+8 n-8 m-8\right) \varphi_i\\
&\coloneqq -B \varphi_i
\end{aligned}
$$
\\

For $ i=m+1,\cdots,n$ 
$$
\begin{aligned}
&4A_N\left(\nabla \varphi, e_i\right)-(n-2) J_N \varphi_i \\
&= \frac{-m(m-1)}{2(n-1)(n-2)}\left(n^2-4 n+8\right) \varphi_i\coloneqq-C \varphi_i 
\end{aligned}
$$

As a result

$$
\int_{N(\Lambda)} P_N \varphi \cdot \varphi=\int(\Delta_N \varphi)^2+B\left|\nabla_M \varphi\right|^2+C\left|\nabla_T \varphi\right|^2+\frac{n-4}{2}Q_N \varphi^2
$$

\noindent\textit{i) Stability}

Direct computation shows that
\begin{equation}
    \begin{aligned}
        \frac{\partial^2}{\partial t^2}E_M(1+tu)&=2(\int_M(\Delta_M u)^2+\frac{m^2-2 m-4}{2}|\nabla_M u|^2-4 Q_M u^2)\\
        &\text{where} \int_M u=0,\quad \|u\|_{L(M)^{2m/(m-4)}}=1\\
        \frac{\partial^2}{\partial t^2}E_N(1+tu)&=2(\int_M(\Delta_N u)^2+B\left|\nabla_M \varphi\right|^2+C\left|\nabla_T \varphi\right|^2-4 Q_N u^2)\\
        &\text{where} \int_{N(\Lambda)} u=0,\quad \|u\|_{L(N(\Lambda))^{2n/(n-4)}}=1\\
    \end{aligned}
\end{equation}

We have the $C>B>\frac{m^2-2m-4}{2}$ and $Q_M> Q_N$ (we leave the proof to \textbf{Lemma \ref{Q-coeff-compare}} in the end of this section). So to show $\frac{\partial^2}{\partial t^2}E_N(1+tu)\geq 0$, it suffices to show that 
\begin{equation}
\int_{N(\Lambda)}(\Delta_N u)^2+\frac{m^2-2 m-4}{2}|\nabla_N u|^2-(\frac{4(m-2)}{m-4}Q_M-\epsilon) u^2\geq 0\label{Q-2ndvariation-middle}
\end{equation}
where $\epsilon$ is small number chosen so that $\frac{m-2}{m-4}Q_M-\epsilon> \frac{n-2}{n-4}Q_N$

The left-hand side of this inequality can be expressed as a bilinear form. Since this form vanishes for pairs of distinct eigenfunctions of $\Delta_N$, it suffices to verify the inequality for eigenfunctions. The same reasoning applies to the expression $\frac{\partial^2}{\partial t^2}E_M(1+tu)$.

Since $Y_4(M,g)$ is minimized by constant, we have $\frac{\partial^2}{\partial t^2}E_M(1+tu)\geq 0$. Let $\lambda_M$ be the 1st eigenvalue for $\Delta_M$, we have 
$$\lambda_M^2+\frac{m^2-2 m-4}{2}\lambda_M\geq \frac{4(m-2)}{m-4}Q_M$$
For small $\Lambda$, we have $\lambda(N(\Lambda))=\lambda_M$; consequently,

$$\lambda(N(\Lambda))^2+\frac{m^2-2 m-4}{2}\lambda (N(\Lambda))\geq \frac{4(m-2)}{m-4}Q_M+\epsilon$$
and (\ref{Q-2ndvariation-middle}) is proved. Then a similar argument as in \textbf{Lemma \ref{lemma-stability}} shows that constant functions are stable critical point for $E_{N(\Lambda)}$.

\noindent \textit{ii) Convergence to a function $u_0$ of small energy}

This part is the same as before, and we get a limiting $u_0$ satisfying
\begin{equation}
    \frac{\int_M|\Delta_M u_0|^2+B|\nabla_M u_0|^2+\frac{n-2}{2}Q_N u_0^2dv_g}{(\int_Mu^{2n/(n-4)})^{(n-4)/n}dv_g}\leq \frac{n-2}{2}Q_NVol(M)^{4/n}\label{Q-smallness}
\end{equation}

\noindent \textit{iii)$u_0=1$}

Consider
    \begin{equation}
        \begin{aligned}
            F_{\lambda,q}(u)&=\frac{\int_M|\Delta u|^2+\frac{m^2-2m-4}{2}|\nabla u|^2+\lambda u^2dv_g}{(\int_Nu^q)^{2/q}dv_g}\\
            s(\lambda,q)&=\inf_{u\in H^1,u\neq 0}F_{\lambda,q}(u)\label{3-lemmaODE-def}
        \end{aligned}
    \end{equation}

The assumption that \((M,g)\) minimizes \(Y_4(g)\) implies that the constant function \(u \equiv 1\) minimizes \(F(\lambda,q)\) for the parameters \((\frac{m-2}{2}Q_M, \frac{2m}{m-4})\). We will show that, under the dimension restriction, \(F(\lambda,q)\) is also minimized by constants for the parameters \((\frac{n-2}{2}Q_N, \frac{2n}{n-4})\). Assuming this holds, note that \(\frac{m^2-2m-4}{2} \leq B\) by \textbf{Lemma \ref{Q-coeff-compare}}. Then inequality (\ref{Q-smallness}) forces \(u_0\) to be constant, completing the proof.

To apply \textbf{Lemma \ref{3-s}}, let \(\bar{g}\) be the volume-normalized metric derived from \(g\), and define
\begin{equation}
e(u) = \frac{1}{\text{Vol}(M)^{4/n}}|\bar{\Delta}_M u|^2 + \frac{m^2-2m-4}{2\,\text{Vol}(M)^{2/n}}|\bar{\nabla}_M u|^2. \notag
\end{equation}
Thus it remains only to verify that \(\frac{n-4}{n}Q_N < \frac{m-4}{m}Q_M\), which is also provided in \textbf{Lemma \ref{Q-coeff-compare}}.

Statements (i), (ii), and (iii) together imply that for any \(\Lambda\), the functional \(Y_4\bigl(N(\frac{1}{2^k}\Lambda),g(\frac{1}{2^k}\Lambda)\bigr)\) is minimized by a constant function when \(k\) is sufficiently large. A comparison argument analogous to the one used near (\ref{3-E-comparison}) will complete the proof. Once more, employing the isometry (\ref{isometry}) and writing \(u_{ij}=\frac{\partial^2 u}{\partial x_i\partial x_j}\), we have
\begin{equation}
    \begin{aligned}
        &\quad\frac{1}{V(\Lambda)}\int_{N(\Lambda)}|\Delta_N u|^2dv_{N(\Lambda)}\\
        &=\int_M\int_{T(\Lambda)}(\Delta_M u+g_{\Lambda}^{ij}u_{ij})^2dxdv_M\\
        &=\int_M\int_{T(\Lambda)}(\Delta_M u)^2+2g_{\Lambda}^{ij}u_{ij}\Delta_M u +g_\Lambda^{ij}g_{\Lambda}^{kl}u_{ij}u_{kl} dxdv_M\\
        &=\int_M\int_{T(\Lambda)}(\Delta_M u)^2+2g_{\Lambda}^{ij}\langle  \nabla_M u_i,\nabla_M u_j \rangle_M+g_\Lambda^{ij}g_{\Lambda}^{kl}u_{ik}u_{jl}dxdv_M
    \end{aligned}
\end{equation}
where we use integration by parts in the last step. To analyze the term \(g_{\Lambda}^{ij}\langle  \nabla_M u_i,\nabla_M u_j \rangle_M\), we choose a local frame \(\{e_\alpha\}\) for \(M\). Then
\[
g_{\Lambda}^{ij}\langle  \nabla_M u_i,\nabla_M u_j \rangle_M
= \sum_{\alpha} g_{\Lambda}^{ij} u_{\alpha i}u_{\alpha j},
\]
which is increasing with respect to the matrix \(g_{\Lambda}^{ij}\).

For the term \(g_\Lambda^{ij}g_{\Lambda}^{kl}u_{ik}u_{jl}\), we note its invariance under \(O(n-m)\). Hence we may assume \(g_{\Lambda}^{-1}\) is diagonal, with eigenvalues \(\lambda_1,\dots,\lambda_{n-m}\). In this basis,
\[
g_\Lambda^{ij}g_{\Lambda}^{kl}u_{ik}u_{jl}
= \sum_{i,j} \lambda_i\lambda_j u_{ij}^2,
\]
which is also increasing in \(g_{\Lambda}^{-1}\). The same comparison argument used near (\ref{3-E-comparison}) therefore applies.

\end{proof}

\begin{lemma}\label{Q-coeff-compare}
    For dimensions \(5 \leq m < n\), the following inequalities hold:
    \begin{equation}
        \begin{aligned}
            &C>B>\frac{m^2-2m-4}{2}\\
            &Q_M> Q_N
        \end{aligned}
    \end{equation}

  And for \(m \geq 6\) or \(n \geq 12\), we have
  \begin{equation}
    \frac{n-4}{n}Q_N<\frac{m-4}{m}Q_M
  \end{equation}
\end{lemma}
\begin{proof}
    \begin{equation}
    B-\frac{m^2-2m-4}{2}=\frac{n-m}{2(n-1)(n-2)}(-mn+6m-4n+4)<0\notag
\end{equation}
And it's easy to see that $B<C$, so $C>B>\frac{m^2-2m-4}{2}$. So the first line is proved.

To prove the second line, we compute
\begin{equation}
    \begin{aligned}
        \quad Q_N-Q_M&=\frac{m(m-n)}{8(n-2)^2 (n-1)^2}\times\\
        &\quad \big[ n^3(m^2 - 4) + n^2(-4m^2 - 5m + 24) + n(16m^2 - 4m - 36)\\
        &\quad -16m^2 + 12m+16\big]
    \end{aligned}
\end{equation} 
Since $n\geq 6$, $m\geq 5$, we have
\begin{equation}
    \begin{aligned}
        n^3(m^2 - 4) + n^2(-4m^2 - 5m + 24)&=n^2\big(n(m^2-4)-4m^2-5m+24\big)\\
        &\geq n^2 (2m^2-5m)>0\notag
    \end{aligned}
\end{equation}
\begin{equation}
    n(16m^2 - 4m - 36) -16m^2 + 12m+16\geq 80m^2-12m-200>0
\end{equation}

To prove the last one, we compute
\begin{equation}
    \begin{aligned}
        &\frac{n-4}{n}Q_N-\frac{m-4}{m}Q_M\\
        &=\frac{(m-n)}{8n(n-2)^2 (n-1)^2}\times\\
        &\quad\big[n^4(m^3 - 4m^2 - 4m + 16) + n^3(-8m^3 + 19m^2 + 40m - 96) \\
        &\quad + n^2(32m^3 - 36m^2 - 132m + 208) + n(-80m^3 + 76m^2 + 160m - 192) \\
        &\quad+ 64(m^3 - m^2 - m + 1)\big]\label{Qn-Qm}
    \end{aligned}
\end{equation}

Let $a(m,n)=(m^3-4m^2-4m+16)n-8m^3+19m^2+40m-96$ and $b(m,n)=(32m^3-36m^2-132m+208)n-80m^3+76m^2+160m-192$. Apparently $64(m^3 - m^2 - m + 1)>0$, so $\frac{n-4}{n}Q_N-\frac{m-4}{m}Q_M<0$ provided $a(m,n)\geq 0$ and $b(m,n)\geq >0$. For $b(m,n)$, since $m\geq 5$

$$32m^2-36m^2 -132m+208\geq 124m^2-132m+208>0$$
we have
\begin{equation}
    \begin{aligned}
        b(m,n)\geq b(m,m+1)&=32m^4-84m^3-92m^2+76m+208\\
        &\geq 76m^3-92m^2+76m+208>0\notag
    \end{aligned}
\end{equation}

Since $m^3-4m^2-4m+16>0$, $a(m,n)\geq a(m,m+1)$. For $m\geq 11$, we have
\begin{equation}
        a(m,m+1)=m^4-11m^3+11m^2+52m-80\geq 11m^2+52m-80>0\notag
\end{equation} 
If $5\leq m\leq 10$, $a(m,n)$ is increasing in $n$, so $a(m,n)\geq 0$ provided 
$$n\geq \frac{8m^3-19m^2-40m+96}{m^3-4m^2-4m+16}\coloneqq c(m)$$
We compute one by one: $c(5)\approx 20.04, c(6)\approx 14.06, c(7)\approx 12.06, c(8)\approx 11.06, c(9)\approx 10.46, c(10)\approx 10.06$. So we are only left these finitely many cases. We plug these finite points into (\ref{Qn-Qm}) and check one by one, we see that $Q_N-Q_M\geq 0$ only for $(m=5,6\leq n\leq 11)$ these 6 cases. 

\end{proof}

It is worth mentioning the work in \cite{GHL}. Define $Y_4^+(M,g)>0$ by (\ref{Q-def}) with the condition $u\in C^{\infty}(M), u> 0$, instead of $u\in H^2(M), u\neq 0$. It is proved that if $Y(M,g)>0$ and $Y_4^+(M,g)>0$, then there exists $\tilde{g}\in [g]$ such that $\tilde{R}>0$ and $\tilde{Q}>0$, provided that $n\geq 6$. The restriction $n\geq 6$ is also an unfortunate byproduct of our technique. It is unclear whether there is a deeper geometric or analytical reason for the failure of the argument in dimension $n=5$ in both Theorem \ref{thm-Q} and \cite{GHL}.

\section{Isoperimetric Ratio Type}

Consider $(M^m, \Sigma=\partial M, h)$.
Given $u$ on $\Sigma$, and let $Pu$ be its hamonic extension to $M$.
\begin{definition}\label{def-isoYamabe}
    For $p_1=\frac{2 (m-1)}{m-2}, q_1=\frac{2 m}{m-2}$, Define
    \begin{equation}
        \begin{aligned}
            E(u)&=\frac{\|P u\|_{L^{q_1}(M)}}{\|u\|_{L^{p_1}(\Sigma)}}\\
            s(M,h)&=\sup\limits_{u\in L^{p_1}(\Sigma), u\neq 0} E(u)
        \end{aligned}
    \end{equation}
\end{definition}

It was shown in \cite{HWY08}\cite{HWY09} that
\begin{theorem}\label{HWY}
For \(1 \leq p < \infty\) and \(1 \leq q \leq \frac{m p}{m-1}\), the operator \(P: L^p(\Sigma) \to L^q(M)\) is bounded, and compact when \(q < \frac{m p}{m-1}\). Hence \(s(M,h)\) is well-defined.

The constant \(S(B^m, h)\) where $h$ is the flat metric is attained by the function \(u_\phi\) induced by a Möbius transformation \(\phi\) (so \(\phi^*(dx^2)=u_\phi^{4m/(m-2)}dx^2\)). If \(s(M,h) < S(B^m, dx^2)\), then \(s(M,h)\) is also attained.
\end{theorem}

Next we prove \textbf{Theorem \ref{thm-isoratio}}.

\begin{proof}
    \hspace{0pt}
    \noindent\textit{1) $u\equiv 1$ is stable critical point for $E_{\bar{\Lambda}}$ when $\Lambda$ is small} 

Let $P_B$ and $P_N$ be the harmonic extension map for $B^m$ and $N(\Lambda)$ respectively. We start with a claim.
 
\noindent \textit{Claim}: For $u\in L^1(S^{m-1})$ or $u\in L^1(N(\Lambda)$, we have
\begin{equation}
    \begin{aligned}
       \int_{S^{m-1}}u=m \int_{\mathbb{B}^m} P_B u\\
   \int_{\Sigma(\Lambda)} u=m\int_{N(\Lambda)} P_{N} v\label{P-L1-estimate}
    \end{aligned}
\end{equation}

For the first equality, let \(\{ \phi_i, \lambda_i \}_{i=0}^{\infty}\) denote the eigenpairs on \(S^{n-1}\), which form an orthonormal basis for \(L^2(S^{n-1})\). In polar coordinates, \(P_B(\phi_i)(r,\theta) = r^k \phi_i(\theta)\) for some \(k \in \mathbb{Z}^+\), and hence $\int_{B^m} P_B(\phi_i)\,dx = 0 \quad \text{for } i \neq 0.$ The factor \(m\) arises because \(P_B(1) = 1\) and \(\operatorname{Vol}(S^{m-1}) = m \operatorname{Vol}(B^m)\). Thus the identity holds for all \(u \in C^\infty(S^{m-1})\). For \(u \in L^1(S^{m-1})\), take a smooth sequence approximating \(u\) and use the fact that \(P : L^p(\Sigma) \to L^q(M)\) is bounded to pass to the limit.

For the second one, let $\bar{u}=\int_{T^{n-m}(\Lambda)} u$ and $\overline{P_Nu}\int_{T^{n-m}(\Lambda)} P_Nu$ be the average of $u$ and $P_Su$ along $T^{n-m}(\Lambda)$. Then
\[
\Delta_{B^m}  \overline{P_Nu}= \int_{T^{n-m}} \Delta_{B^m} (P_N u)= -\int_{T^{n-m}} \Delta_{T^{n-m}} (P_N v)= 0 .
\]
Thus $P_B \bar{u} = \overline{P_N(u)}$, and in (\ref{P-L1-estimate}) the second equality follows from the first one.

Next we compute $\frac{\partial^2}{\partial t^2} E_B(1+t u)$ and $\frac{\partial^2}{\partial t^2} E_N(1+t u)$ for $\int_{S^{m-1}} u=0, \int_{\Sigma(\Lambda)} u=0$.

$$
\begin{aligned}
\frac{\partial^2}{\partial t^2} E_B (1+t u)|_{t=0}&=C(m)\left((m+2) \int_{B^m} (P_B u)^2-\int_{S^{n-1}} u^2\right)\\
\frac{\partial^2}{\partial t^2} E_N(1+t u)|_{t=0}&=C(m,n, \Lambda)\left(\frac{m(n+2)}{n} \int_N\left(P_N u\right)^2-\int_{\Sigma} u^2\right)
\end{aligned}
$$
and $\int_{B^m}P_Bu=0$, $\int_{N}P_Nu=0$ are used in the computation.

On a general $(M,\Sigma,h)$, by \textbf{Theorem \ref{HWY}}, $P: L^2(\Sigma) \rightarrow L^2(M)$ is compact, we can consider the variation problem

$$
\mu_i(M)=\inf\limits_{\text{codim} V=i} \sup _{ u \in V,u\neq 0}\frac{\|P u\|_{L^2(M)}}{\|u\|_{L^2(\Sigma)}}
$$

$\mu_i(M)$ can always achieved by a smooth function $\phi_i$ satisfying

$$
P^* P \phi_i=\int_M K(x, \varphi) P u_i(x) d x=\mu_i(M) \phi_i \quad x \in M, \varphi \in \Sigma \text {, }
$$
where $K(x, \varphi)$ is the Poison kernel and $P^*$ is adjoint operator for $P$. $P^* P: L^2(\Sigma) \rightarrow L^2(\Sigma)$ is a compact. setf-adjoint operator,
so by standard theory (see \cite{Evans} D4, for example), its eiganvalue is a sequence decreasing to 0 .

Applying the same strategy, there exists $k$ large st $\mu_1(N(\Lambda/2^k))=\mu_1\left(B^m\right)$ and the stability follows since $\frac{m(n+4)}{n}<m+2$ for $n>m$.

\noindent ii) convergence to $u_0$ st $E_{B^m}\left(u_0\right) \geqslant E_{B^m}(1)$

We need the following Concentration-Compactness lemma proved in \cite{HWY09}:
\begin{lemma}(Concentration compactness lemma). Assume $m \geqslant 2,\left(M^m, h\right)$ is a smooth compact Riemannian manifold with boundary $\Sigma=\partial M, 1<p<\infty, f_i \in L^p(\Sigma)$ such that $f_i \rightharpoonup f$ in $L^p(\Sigma)$. After passing to a subsequence assume

$$
\left|f_i\right|^p d S \rightharpoonup \sigma \quad \text { in } \mathcal{M}(\Sigma), \quad\left|P f_i\right|^{\frac{n p}{n-1}} d \mu \rightharpoonup \nu \quad \text { in } \mathcal{M}(M) \text {. }
$$

Here $\mathcal{M}(\Sigma)$ is the space of all Radon measures on $\Sigma$. Then we have

\noindent- $\left.\nu\right|_{M \backslash \Sigma}=|P f|^{\frac{n p}{n-1}} d \mu$. Moreover for every Borel set $E \subset \Sigma, \nu(E)^{\frac{n-1}{n p}} \leqslant c_{n, p} \sigma(E)^{\frac{1}{p}}$.

\noindent- There exists a countable set of points $\zeta_j \in \Sigma$ such that $\nu=|P f|^{\frac{n p}{n-1}} d \mu+\sum_j \nu_j \delta_{\zeta_j}, \sigma \geqslant|f|^p d S+\sum_j \sigma_j \delta_{\zeta_j}$, here $\sigma_j=\sigma\left(\left\{\zeta_j\right\}\right)$ and $\nu_j^{\frac{n-1}{n p}} \leqslant c_{n, p} \sigma_j^{\frac{1}{p}}$.
\end{lemma}
Apply the same strategy, and we get a limiting $u_0$ so thet
\begin{equation}
\|P_B u\|_{L^{2n/(n-2)}(B^m)} \leqslant \frac{|B^m|^{(n-2)/2n}}{|S^{m-1}|^{(n-2)/(2(n-1))}} \|u\|_{L^{2(n-1)/(n-2)}(S^{m-1})}\label{isoYamabe-smallenergy}
\end{equation}

\noindent iii) $u_0=1$

We are going to use Riesz-Thorin interpolation theory \textbf{Lemma \ref{RieszThorin}}. For simplicity, in this part we take normalized $L_p$ norm
$$
\|u\|_{L(M)}=\big(\frac{1}{|M|} \int_M |u|^p\big)^{\frac{1}{p}} \text { where } M=S^{m-1} \text { or } B^m
$$

Consider P: $L^{1}\left(S^m\right) \rightarrow L^{1}\left(B^m\right)$
if $u \geq 0$, by the (\ref{P-L1-estimate}), we have.
$$
\|P u\|_{L^{1}\left(B^{m}\right)}=\|u\|_{L^{1}\left(S^{m-1}\right)}
$$

If $u$ changes sign, $|P u| \leq P|u|$ (Poison bernel is non-negative), so
$$\|P u\|_{L^{1}\left(B^m\right)} \leq\|P(|u|)\|_{L^{1}\left(B^m\right)} \leq\|u\|_{L^{1}\left(S^{m-1}\right)}
$$
Use Riesz-Thorin interpolation for $P$:

$$
\begin{array}{ll}
P: L^{1}\left(S^{m-1}\right) \rightarrow L^{1}\left(B^m\right) & \text { norm } 1 \\
P: L^{p_1}\left(S^{m-1}\right) \rightarrow L^{q_1}\left(B^m\right) & \text { norm } 1 \quad p_1=\frac{2(m-1)}{m-2}, q_1=\frac{2 m}{m-2}
\end{array}
$$

Let $p_\theta=\frac{2(n-1)}{n-2}$. Pick $\theta$ st $\frac{1}{p_\theta}=\theta+(1-\theta) \frac{1}{p_1}$
and we find $\theta=\frac{n-m}{m(n-1)}$.
Then we compute
$$
 \frac{1}{q_\theta}=\theta+(1-\theta) \frac{1}{q_1}=\frac{n m^2-m n-2 m^2+2 n}{2 m^2(n-1)}
$$

By Riesz-Thorin interpolation
$$
\|P u\|_{L^{q_\theta}(B^m)} \leqslant\|u\|_{L^{p_\theta}(S^{m-1})}
$$
By H\'older inequality, if $q_{\theta}> \frac{2 n}{n-2}$, then we have
$$
\|P u\|_{L^{2n/(n-2)}(B^m)} \leqslant\|u\|_{L^{2(n-1)/(n-2)}(S^{m-1})}
$$
with equality iff $u$ is constant, which implies $u_0=1$ by (\ref{isoYamabe-smallenergy}).

Compute $q_\theta-\frac{2 n}{n-2}$:
$$
q_\theta-\frac{2 n}{n-2}=(n-m) \frac{(m-2) n-2 m}{(n-2)\left(\left(m^2-m+2\right) n-2 m^2\right)}
$$

$\left(m^2-m+2\right) n-2 m^2>0$, and $(m-2) n-2 m>0$ unless $m=3, n=4,5,6$.
\end{proof}

\begin{remark}
On the model space \((B^m, S^{m-1}, dx^2)\), let \(\{(u_i, \lambda_i)\}\) denote the eigenpairs of \(-\Delta_{S^{m-1}}\). Then each \(u_i\) is also an eigenfunction of \(P^*P\).

To verify this, consider two distinct eigenfunctions \(u_i\) and \(u_j\) (\(i \neq j\)). In polar coordinates, \(P u_i = r^k u_i\) and \(P u_j = r^l u_j\) for some integers \(k, l \ge 0\). Using orthogonality,
\[
\begin{aligned}
\int_{S^{m-1}} u_i\, (P^*P u_j) \, d\varphi 
&= \int_{B^m} (P u_i)(P u_j) \, dx \\
&= \int_0^1 r^{k+l}\, dr \int_{S^{m-1}} u_i u_j \, d\varphi = 0 .
\end{aligned}
\]

Since \(\{u_i\}_{i=0}^\infty\) forms an orthonormal basis of \(L^2(S^{m-1})\), it follows that each \(u_i\) is an eigenfunction of \(P^*P\). Moreover, the first non‑zero eigenvalue satisfies
\[
\mu_1(B^m) = \frac{1}{m+2}.
\]
\end{remark}

\begin{remark}
Let \( p^* = \frac{mp}{p-1} \) for \( p \geq 1 \). By \textbf{Theorem \ref{HWY}}, for any \( q \leq p^* \) we have
\[
\|Pu\|_{q} \leq C(p,q,m)\,\|u\|_p,
\]
where \( C(p,q,m) \) is a finite constant.

According to Corollary 3.2 of \cite{HWY09}, for \( p \geq \frac{2(m-1)}{m-2} \) and \( q = p^* \), the constant \( C(p,p^*,m) \) is attained by a constant function. For \( 1 < p < \frac{2(m-1)}{m-2} \), however, \( C(p,p^*,m) \) is not attained by a constant, because \( u \equiv 1 \) fails to be a stable critical point. Step (iii) shows that for each such \( p \) there exists a critical exponent \( \bar{q}(p) \) such that \( C(p,q,m) \) is attained by a constant whenever \( q \leq \bar{q}(p) \).

It would be interesting to determine precisely when \( C(p,q,m) \) is attained by a constant, and to compute \( C(p,p^*,m) \) for \( 1 \leq p < \frac{2(m-1)}{m-2} \).
\end{remark}

\begin{lemma} \label{RieszThorin}(Riesz-Thorin) $(\Omega, \mu),(\tau, \nu)$ measure spaces, $p_0, p_1, q_0, q_1 \in \left[1, \infty\right)$
$$
\frac{1}{p_\theta}=\frac{1-\theta}{p_0}+\frac{\theta}{p_1}, \frac{1}{q_\theta}=\frac{1-\theta}{q_0}+\frac{\theta}{q_1} \quad \theta \in\left[0, 1\right]
$$
If $T: L^{p_0}(\mu) \cap L^{p_1}(\mu) \rightarrow L^{q_0}(\nu) \cap L^{q_1}(\nu)$ st $\|T\|_{L^{p_i} \rightarrow L^{q_i}} \leq C_i \quad i=0,1$ then $T \cdot L^{p_\theta}(\mu) \rightarrow L^{q_\theta}(\nu)$ is bounded
 \[\|T\|_{L^{p_\theta} \rightarrow L^{q_\theta}} \leqslant C_0^{1 \cdot \theta} C_1^\theta\]
\end{lemma}

\section{The Yamabe Constant for $M\times S^1(T)$ and The Yamabe Invariant}

In this section we post a question on the Yamabe constant for $M\times S^1(T)$ for arbitrary $T> 0$, and some consequences on the Yamabe invariants of $S^2\times S^2$ if this question has a positive answer.

\subsection{Model}
\,

The model we consider is from \cite{Schoen2}: \((S^{n-1} \times S^1(T),\, g_{0}(T) = h_{std} + dt^2)\), where \(h_{std}\) denotes the standard round metric on \(S^{n-1}\) with scalar curvature \((n-1)(n-2)\). Let \(u_T\) be a minimizer for the Yamabe constant of this manifold.

Via the covering map, the conformal metric \(u_T^{4/(n-2)}g_T\) lifts to a metric \((u_T^*)^{4/(n-2)}g_T^*\) on the cylinder \(\mathbb{R} \times S^{n-1}\). Because \(\mathbb{R} \times S^{n-1}\) is conformally equivalent to \((\mathbb{R}^n \setminus \{0\}, dx^2)\), this lifted metric can be expressed as \(v^{\frac{2n}{n-2}}dx^2\) on \(\mathbb{R}^n \setminus \{0\}\), where \(v\) extends to a singular function at the origin.

By \cite{CGS} the constant scalar curvature equation forces \(v\) to be rotationally symmetric; consequently, the original minimizer \(u_T\) depends only on the \(S^1(T)\) factor and satisfies the ODE
\begin{equation}
u'' - \frac{(n-2)^2}{4}\,u + \frac{n(n-2)}{4}\,u^{(n+2)/(n-2)} = 0.
\label{Yamabe-ODE}
\end{equation}
Let \(u_{n,T}\) be the \(T\)-periodic solution of (\ref{Yamabe-ODE}) that minimizes \(Y(g_0(T))\), and \(u_{n,\infty}(t) = (\cosh t)^{-(n-2)/2}\) the solution on the infinite cylinder \(\mathbb{R} \times S^{n-1}\), which are unique up to translation along \(S^1(T)\) or $\mathbb{R}$. The metric \((S^{n-1} \times \mathbb{R},\, u_{n,\infty}^{4/(n-2)}(h_{std}+dt^2))\) is isometric to the standard round sphere \((S^n,h_{std})\).

Taking the limit \(T \to \infty\), we have \(u_{n,T} \to u_{n,\infty}\) and consequently
\[
\lim_{T \to \infty} Y(S^{n-1} \times S^1(T),g_0(T)) = Y(S^n,h_{std})\Longrightarrow \sigma(S^{n-1} \times S^1)=\sigma(S^n)
\]

Although \textbf{Theorem \ref{thm_Yamabe}} assumes \(m \geq 3\), the construction described above applies to all \(m = n-1 \geq 2\), including the case of \(S^2 \times S^1\).

\subsection{Conjecture on the Yamabe constant of $M\times S^1$ and Yamabe Invariant of $S^2\times T^2$}
\,

Next, consider a general Yamabe metric \((M^{n-1},h)\) and the product manifold
\[
(N(T) = M \times S^1(T),\; g(T) = h + d\theta^2),
\]
where \(T > 0\) is an arbitrary period. Recall (\ref{isometry-Yamabe}). Intuitively, as \(T\) grows, the term \(A(u)\) becomes less significant while \(B(u)\) becomes dominant. Hence, minimizing \(E_{g(T)}\) should prioritize making \(B(u)\) small, with \(A(u)\) treated as a secondary concern as $T$ grows.

Indeed, let \(u_T\) denote a minimizer of \(E_{g(T)}\). Then
\begin{equation}
A(u_T)\ \text{is non‑decreasing},\qquad B(u_T)\searrow 0 \quad \text{as } T \to \infty.
\end{equation}
This follows directly from the comparison inequalities
\[
E_{g(T_1)}(u_{T_1}) \le E_{g(T_1)}(u_{T_2}), \qquad
E_{g(T_2)}(u_{T_2}) \le E_{g(T_2)}(u_{T_1}) \qquad (T_1 < T_2),
\]
together with the fact that \(Y(g(T)) < Y(S^n)\).

By \textbf{Lemma \ref{3-s}}, define the \(M\)-directional average
\[
\bar{u} = \frac{1}{\operatorname{Vol}(M)}\Bigl(\int_M u^{2n/(n-2)} \, dv_h\Bigr)^{(n-2)/(2n)} .
\]
Then
\[
\begin{aligned}
\int_M \bigl(c_n|\nabla_M \bar{u}|^2 + R_h \bar{u}^2\bigr)\, dv_h
   &\leq \int_M \bigl(c_n|\nabla_M u|^2 + R_h u^2\bigr)\, dv_h, \\[4pt]
\int_{S^1}\int_M \bar{u}^{2n/(n-2)} \, dv_h\, dt
   &= \int_{S^1}\int_M u^{2n/(n-2)} \, dv_h\, dt,
\end{aligned}
\]
and hence \(B(\bar{u}) \le B(u)\). This shows that to reduce \(B\), one should average along the \(M\)-direction.

According to \textbf{Theorem \ref{thm_Yamabe}}, when \(T\) is small, \(Y(g(T))\) is attained by a constant function—which trivially does not depend on the \(M\)-direction. As \(T\) increases, the term \(B(u)\) gains weight in the functional \(E_{g(T)}\). Therefore, in order to further minimize \(E_{g(T)}\), one should devote even more effort to making \(B(u)\) small, i.e., to averaging along \(M\).

This observation leads to the following conjecture:

\begin{conjecture}\label{conj}
    Let $(M^{n-1},h)$ be a Yamabe metric and $Y(h)>0$ and normalized to have $R_h\equiv (n-1)(n-2)$, then $Y(g(T))$ is minimized by $u_{n,T}$ from (\ref{Yamabe-ODE}) that only depend on $t$, and in particular,
    \begin{equation}
        \begin{aligned}
        \frac{Y(h+dt^2)}{Y(h_{std}+dt^2)}&=(\frac{Y(h)}{Y(h_{std})})^{(n-1)/n}\\
        \Longrightarrow\frac{\sigma(M^{n-1}\times S^1)}{\sigma(S^n)}&\geq(\frac{Y(h)}{Y(h_{std})})^{(n-1)/n}
        \end{aligned}\label{conj-general}
    \end{equation}
    As a consequence
    \begin{equation}
        \sigma(S^{2}\times T^2)=\sigma(S^4)\label{conj-T^2}
    \end{equation}
\end{conjecture}
If the first part of the conjecture holds, then \(R_h = (n-1)(n-2)\) implies $Y(h) = (n-1)(n-2)\,|M|^{2/(n-1)}$. Moreover,
\begin{equation}
Y(g(T)) = E_{g(T)}(u_{n,T}) 
= |M|^{2/n}\,
   \frac{\displaystyle\int_{S^1(T)}\!\bigl[c_n(\partial_t u_{n,T})^2 + (n-1)(n-2)u_{n,T}^2\bigr]}
        {\left(\displaystyle\int_{S^1(T)} u_{n,T}^{2n/(n-2)}\right)^{(n-2)/n}} .
\tag{1}
\end{equation}
Replacing \(g(T)\) by \(g_0(T)\) and forming the ratio of the two expressions yields (\ref{conj-general}).

For (\ref{conj-T^2}), consider the metric
\begin{equation}
g(T_1,T_2) = u_{4,T_2}^{\,2}(s)
            \Bigl(ds^2 + u_{3,T_1}(t)^{\,4}\bigl(dt^2 + h_{std}\bigr)\Bigr).\label{Yamabe-construction}
\end{equation}
In $u_{n,T}$, $n$ deontes dimension and $T$ is period. \(u_{3,T_1}^{\,4}(dt^2 + h_{std})\) is a Yamabe metric whose Yamabe constant approaches \(Y(S^3,h_{std})\) as \(T_1 \to \infty\). Sending \(T_2 \to \infty\) and applying (\ref{conj-general}) then gives (\ref{conj-T^2}).

This conjecture is also motivated by work on the Yamabe constant (or Yamabe invariant) of product manifolds—a topic that has been extensively studied and will be discussed in the following subsection.

It is also relevant to mention Rosenberg’s \(S^1\)-stability conjecture, which asserts that a closed manifold \(M\) admits a metric of positive scalar curvature (PSC) if and only if \(M\times S^1\) does, or equivalently \(\sigma(M)>0 \iff \sigma(M\times S^1)>0\) \cite{Rosen}. The implication \(\sigma(M)>0 \Rightarrow \sigma(M\times S^1)>0\) is trivial, and (\ref{conj-general}) provides a stronger quantitative statement. An interesting question is whether \(\sigma(M\times S^1)\) gives a universal lower bound for \(\sigma(M)\) in some natural way.

\subsection{Conjecture on Yamabe Invariant of $S^2\times S^2$}

\,

It is well known that the saddle point of the Yamabe invariant \(\sigma(M)\) is closed related to a metric of constant Ricci curvature. Consider \(S^2 \times S^2\) with two round metrics \(h_{\text{std},1}\) and \(h_{\text{std},2}\). The product metric \(h_1 + h_2\) has constant Ricci curvature, yet surprisingly,
\[
Y(h_{\text{std},1} + h_{\text{std},2}) < \sigma(S^2 \times S^2).
\]
Since \(h_{\text{std},1} + h_{\text{std},2}\) is Einstein and distinct from the round metric on \(S^4\), Obata’s theorem \cite{Obata} implies it is the unique CSC metric in its conformal class. Moreover, by the result of \cite{BWZ}, any CSC metric sufficiently close to \(h_{\text{std},1} + h_{\text{std},2}\) is also the only CSC metric within its conformal class, hence a Yamabe minimizer.

Applying this to the family \(h_{\text{std},1} + r^2 h_{\text{std},2}\) with \(r\) near \(1\) (but \(r \neq 1\)), a direct computation gives
\[
Y(h_{\text{std},1} + h_{\text{std},2}) \;<\; Y(h_{\text{std},1} + r^2 h_{\text{std},2}) \;\le\; \sigma(S^2 \times S^2).
\]

The inequality above suggests passing to the limit \( r \to \infty \). According to \cite{AFP}, if \((M_1^{m_1},h_1)\) and \((M_2^{m_2},h_2)\) satisfy \(m_1, m_2 \ge 2\) and \(h_1\) has positive scalar curvature, then
\[
\lim_{r \to \infty} Y(M_1 \times M_2,\, h_1 + r^2 h_2)
   = Y\bigl(M_1 \times \mathbb{R}^{m_2},\, h_1 + g_{\text{Flat}}\bigr),
\]
where the right‑hand side is the infimum taken over compactly supported functions in \(H^1(M_1 \times \mathbb{R}^{m_2})\).

Recall a conjecture from
\begin{conjecture}[K.~Akutagawa]
Let \((M_1^{m_1},h_1)\) be a Yamabe metric of positive scalar curvature and let \(m_2 \ge 2\). Then the infimum defining
\(Y(M_1 \times \mathbb{R}^{m_2},\, h_1 + g_{\text{Flat}})\) is achieved by a radial function on \(\mathbb{R}^{m_2}\).
\end{conjecture}

Assuming this conjecture holds, the result of \cite{AFP} would imply
\[
\sigma(S^2 \times S^2) \;\ge\; Y\bigl(S^2 \times \mathbb{R}^2,\, h_{\text{std},1} + g_{\text{flat}}\bigr)
                     \;\approx\; 59.40
                     \;>\; Y(h_{\text{std},1} + h_{\text{std},2}) = 16\pi .
\]

It is also worth mentioning the following result of \cite{Petean09}.

\begin{theorem}[J.~Petean]
Suppose \((M^{n-1},h)\) satisfies \(\operatorname{Ric}_h \ge (n-2)h\). Then
\begin{equation}
\sigma(M\times \mathbb{R}) \;\ge\; Y(M\times \mathbb{R},\,h+dt^2)
\;\ge\; \left(\frac{\operatorname{Vol}(M,h)}{\operatorname{Vol}(S^{n-1},h_{\text{std}})}\right)^{2/n} \sigma(S^n).
\end{equation}
\end{theorem}

This theorem provides a partial affirmative answer to \textbf{Conjecture \ref{conj}}.

If (\ref{conj-T^2}) holds, one might further conjecture that \(\sigma(S^2 \times S^2) = \sigma(S^4)\). Intuitively, the sphere \(S^2\) is more "positively curved" than the torus \(T^2\), so we would expect \(\sigma(S^2 \times S^2) \geq \sigma(S^2 \times T^2)\).

This intuition is supported by the result of \cite{Le99}, which states that for a closed Kähler surface \(M\),
\[
\begin{cases}
\sigma(M) > 0 & \Longleftrightarrow \operatorname{Kod}(M) = -\infty,\\[4pt]
\sigma(M) = 0 & \Longleftrightarrow \operatorname{Kod}(M) = 0,\,1,\\[4pt]
\sigma(M) < 0 & \Longleftrightarrow \operatorname{Kod}(M) = 2,
\end{cases}
\]
where \(\operatorname{Kod}(M)\) denotes the Kodaira dimension. Since the Kodaira dimension is additive,
\[
\operatorname{Kod}(X \times Y) = \operatorname{Kod}(X) + \operatorname{Kod}(Y),
\]
and because \(\operatorname{Kod}(S^2) = -\infty\) while \(\operatorname{Kod}(T^2) = 0\), it is natural to expect \(\sigma(M^2 \times S^2) \geq \sigma(M^2 \times T^2)\).

The guess that \(\sigma(S^2\times S^2) = \sigma(S^4)\) is also supported by the construction (\ref{Yamabe-construction}). Restrict $g(T_1,T_2)$ to the 2‑torus:
\begin{equation}
\Bigl(S^1(T_1)\times S^1(T_2),\;
\bar{g}(T_1,T_2)=u_{4,T_2}^{\,2}(s)
\bigl(ds^2+u_{3,T_1}^{\,4}(t)dt^2\bigr)\Bigr).
\label{Yamabe-construction-2}
\end{equation}

Let \(C(s_0)=\{(s_0,t)\mid t\in S^1(T_1)\}\) be the circle at \(s=s_0\). Because \(u_{n,T}\to u_{n,\infty}=(\cosh t)^{-(n-2)/2}\) as \(T\to\infty\), the length of \(C(s)\) satisfies
\[
\frac{1}{C}\,u_{4,T_2}(s)\;\le\;
L\bigl(C(s)\bigr)=u_{4,T_2}(s)\int_{S^1(T_1)}u_{3,T_1}^2(t)\,dt
\;\le\; C\,u_{4,T_2}(s)
\]
for some constant \(C>0\) independent of \(T_2\).

Moreover, \(\displaystyle\lim_{T_2\to\infty}\inf_{s} u_{4,T_2}(s)=0\); hence the metric \(\bar{g}(T_1,T_2)\) describes a torus that develops an arbitrarily thin neck as \(T_2\to\infty\). By cutting along this shrinking neck and gluing two spherical caps, holefully one obtains a metric \(\tilde{g}(T_1,T_2)\) on \(S^2\times S^2\) such that
\[
Y\bigl(\tilde{g}(T_1,T_2)\bigr)\;\xrightarrow[T_2\to\infty]{}\;Y(S^4).
\]


\begin{thebibliography}{999}     


\bibitem{AFP}
Akutagawa K, Florit L A, Petean J.
\newblock On Yamabe constants of Riemannian products.
\newblock {\em Communications in analysis and geometry}, 15(5): 947-969, 2007.

\bibitem{akutagawa2021yamabe}
Akutagawa, Kazuo.
\newblock The Yamabe invariant.
\newblock {\em Sugaku Expositions}, 34(1):1-34, June 2021.
\newblock Article electronically published on April 28, 2021.
\newblock doi:10.1090/suga/456.

\bibitem{Aubin}
Aubin T.
\newblock Équations différentielles non linéaires et problème de Yamabe concernant la courbure scalaire.
\newblock {\em J. Math. Pures Appl.(9)}, 55: 269-296, 1976.

\bibitem{BF}Bidaut-Véron M F, Véron L. Nonlinear elliptic equations on compact Riemannian manifolds and asymptotics of Emden equations[J]. Inventiones mathematicae, 1991, 106(1): 489-539.

\bibitem{BE}
Brézis, Haïm, and Elliott Lieb.
\newblock A relation between pointwise convergence of functions and convergence of functionals.
\newblock {\em Proceedings of the American Mathematical Society}, 88(3): 486-490, 1983.

\bibitem{BWZ}
Bohm, Christoph, McKenzie Wang, and Wolfgang Ziller.
\newblock A variational approach for compact homogeneous Einstein manifolds.
\newblock {\em Geometric and Functional Analysis}, 14(4): 681-733, 2004.

\bibitem{CGS}
Caffarelli, Luis et al.
\newblock Asymptotic symmetry and local behavior of semilinear elliptic equations with critical sobolev growth.
\newblock {\em Communications on Pure and Applied Mathematics}, 42: 271-297, 1989.

\bibitem{CWW}
Chen, X., Wei, W., \& Wu, N.
\newblock Almost sharp Sobolev trace inequalities in the unit ball under constraints.
\newblock {\em Advances in Mathematics}, 459: 110023, 2024.

\bibitem{CLW}
Chen, X.; Lai, M.; Wang, F.
\newblock Escobar-Yamabe compactifications for Poincaré-Einstein manifolds and rigidity theorems.
\newblock {\em Advances in Mathematics}, 343: 16-35, 2019.

\bibitem{Evans}
Evans, L. C.
\newblock Partial differential equations.
\newblock American Mathematical Society, 2022.

\bibitem{Fe}
Federer, H.
\newblock Geometric Measure Theory.
\newblock Grundlehren Math. Wiss., vol. 153, Springer, New York, 1969, xiv+676 pp.

\bibitem{GHL}Gursky M J, Hang F, Lin Y J. Riemannian manifolds with positive Yamabe invariant and Paneitz operator[J]. International Mathematics Research Notices, 2016, 2016(5): 1348-1367.

\bibitem{HWY09}
Hang, F., Wang, X., Yan, X.
\newblock An integral equation in conformal geometry.
\newblock {\em Annales de l'Institut Henri Poincaré C, Analyse non linéaire}, 26(1): 1-21, 2009.

\bibitem{HWY08}
Hang, F., Wang, X., Yan, X.
\newblock Sharp integral inequalities for harmonic functions.
\newblock {\em Communications on Pure and Applied Mathematics}, 61(1): 54-95, 2008.

\bibitem{HY}Hang F, Yang P C. Lectures on the fourth-order Q curvature equation[J]. Lect. Notes Ser. Inst. Math. Sci. Natl. Univ. Singap, 2016, 31: 1-33.

\bibitem{Ko85-1}
Kobayashi, Osamu.
\newblock On large scalar curvature.
\newblock Research Report 11, Dept. Math., Keio Univ., 1985.

\bibitem{Ko85-2}
Kobayashi, Osamu.
\newblock On a conformally invariant functional of the space of Riemannian metrics.
\newblock {\em Journal of the Mathematical Society of Japan}, 37(3): 373--389, 1985.
\newblock DOI: 10.2969/jmsj/03730373.
\newblock MR: 792982.

\bibitem{Le99}
Lebrun, C.
\newblock Kodaira dimension and the Yamabe problem.
\newblock {\em Communications in Analysis and Geometry}, 7: 133-156, 1999.

\bibitem{Obata}
Obata, Morio.
\newblock The conjectures on conformal transformations of Riemannian manifolds.
\newblock {\em Journal of Differential Geometry}, 6: 247-258, 1971.

\bibitem{Petean09}
Petean, Jimmy.
\newblock Isoperimetric regions in spherical cones and Yamabe constants of M× S.
\newblock {\em Geometriae Dedicata}, 143(1): 37-48, 2009.

\bibitem{Rosen}
Rosenberg, J.
\newblock Manifolds of positive scalar curvature: a progress report.
\newblock {\em Surveys in Differential Geometry}, 11(1): 259-294, 2006.


\bibitem{Schoen}
Schoen, R.
\newblock Conformal deformation of a Riemannian metric to constant scalar curvature.
\newblock {\em Journal of Differential Geometry}, 20(2): 479-495, 1984.

\bibitem{Schoen2}
Schoen, Richard M.
\newblock Variational theory for the total scalar curvature functional for Riemannian metrics and related topics.
\newblock (1989).

\bibitem{SM}
Struwe, Michael, and M. Struwe.
\newblock Variational methods.
\newblock Vol. 991, Springer-Verlag, Berlin, 2000.

\bibitem{Trudinger}
Trudinger, N. S.
\newblock Remarks concerning the conformal deformation of Riemannian structures on compact manifolds.
\newblock {\em Annali della Scuola Normale Superiore di Pisa-Scienze Fisiche e Matematiche}, 22(2): 265-274, 1968.

\bibitem{Yamabe}
Yamabe, H.
\newblock On a deformation of Riemannian structures on compact manifolds.
\newblock 1960.

\end{thebibliography}
\end{document}